\documentclass[11pt]{amsart}
\usepackage[latin1]{inputenc}
\usepackage{amsmath,amsthm,amssymb}
\usepackage{amsfonts}
\usepackage{amsmath,amsthm,amssymb,amscd}
\usepackage{latexsym}
\usepackage{color}
\usepackage{graphicx}
\usepackage{mathrsfs}
\usepackage{comment}
\usepackage{cancel}

\makeatletter \@addtoreset{equation}{section} \makeatother

\newtheorem{theorem}{Theorem}[section]

\newtheorem{proposition}{Proposition}[section]
\newtheorem{lemma}{Lemma}[section]
\newtheorem{remark}{Remark}[section]

\newtheoremstyle{case}{}{}{}{}{}{\em:}{ }{}
\theoremstyle{case}
\newtheorem{case}{\it Case}
\newtheorem{subcase}{\it Subcase}
\numberwithin{subcase}{case}

\allowdisplaybreaks

\begin{document}

%%%%%%%%%%%%%%%%%%%%%%%%%%%%%%%%%%%%%%%%%%%%%%%%%%%%%%

\title[Kirchhoff type double phase problems]
{Existence and multiplicity results for Kirchhoff type problems on a double phase setting}

\author[Alessio Fiscella]{Alessio Fiscella}
\email{fiscella@ime.unicamp.br}

\author[Andrea Pinamonti]{Andrea Pinamonti}
\email{andrea.pinamonti@unitn.it}

\address[Alessio Fiscella]{Departamento de Matem\'atica, Universidade Estadual de Campinas, IMECC\\
Rua S\'ergio Buarque de Holanda, 651, Campinas, SP CEP 13083--859 Brazil}
\address[Andrea Pinamonti]{Dipartimento di Matematica,	Universit\`a degli Studi di Trento,
	Via Sommarive 14, 38123, Povo (Trento), Italy}

\subjclass[2010]{35J62; 35J92; 35J20.} \keywords{Kirchhoff coefficients; double phase problems; variational methods.}

\date{July 2, 2020}
\maketitle

%%%%%%%%%%%%%%%%%%%%%%%%%%%%%%%%%%%%%%%%%%%%%%%%%%%%%%%%%%%%%%%%%%%%%%%%

\begin{abstract} 
In this paper, we study two classes of Kirchhoff type problems set on a double phase framework. That is, the functional space where finding solutions coincides with the Musielak-Orlicz-Sobolev space $W^{1,\mathcal H}_0(\Omega)$, with modular function $\mathcal H$ related to the so called double phase operator. Via a variational approach, we provide existence and multiplicity results.

\end{abstract}

\maketitle

\section{Introduction}\label{sec:introduction}

Recently, a great attention has been devoted to the study of the energy functional
\begin{equation}\label{funzionale}
u\mapsto\int_\Omega\left(|\nabla u|^p+a(x)|\nabla u|^q\right)dx\quad\mbox{with }1<p<q,\quad a(\cdot)\geq0
\end{equation}
whose integrand switches between two different types of elliptic rates according to the coefficient $a(\cdot)$. This kind of functional was introduced by Zhikov in \cite{Z1,Z2,Z3,ZK} in order to provide models for strongly anisotropic materials.
Also, \eqref{funzionale} falls into the class of functionals with non-standard growth conditions, according to Marcellini's definition given in \cite{M1,M2}. Following this direction, Mingione et al. provide different regularity results for minimizers of \eqref{funzionale} in \cite{BCM,BCM2,CM1,CM2}. In \cite{CS}, Colasuonno and Squassina analyze the eigenvalue problem with Dirichlet boundary condition of the double phase operator $\mbox{div}\left(|\nabla u|^{p-2}\nabla u+a(x)|\nabla u|^{q-2}\nabla u\right)$, whose Euler-Lagrange functional corresponds to \eqref{funzionale}.
While, for existence and multiplicity of solutions of nonlinear problems driven by the double phase operator, we refer to \cite{GLL,LD,PS}, with the help of variational techniques, and to \cite{GW1,GW2}, through a non-variational characterization.

Aim of the present paper is to study different classes of variational Kirchhoff type problems, set on a double phase framework which will be discussed in detail on Section \ref{sec2}. For this, we first introduce the following problem 
\begin{equation}\label{P}
\left\{\begin{array}{ll}
-M\left[\displaystyle\int_\Omega\left(\frac{|\nabla u|^p}{p}+a(x)\frac{|\nabla u|^q}{q}\right)dx\right]
\mbox{div}\left(|\nabla u|^{p-2}\nabla u+a(x)|\nabla u|^{q-2}\nabla u\right)=f(x,u) & \mbox{in } \Omega,\\
u=0 & \mbox{in } \partial\Omega,
\end{array}
\right.
\end{equation}
where {\em along the paper, and without further mentioning,} $\Omega\subset\mathbb R^N$ is an open, bounded set with Lipschitz boundary, $N\geq2$, $1<p<q<N$ and
\begin{equation}\label{cruciale}
\frac{q}{p}<1+\frac{1}{N},\qquad a:\overline{\Omega}\to[0,\infty)\mbox{ {\em is Lipschitz continuous.}}
\end{equation}
Here, we assume that $M:[0,\infty)\to[0,\infty)$ is a {\em continuous} function verifying:
\begin{enumerate}
\item[$(M_1)$]
\textit{there exists} $\theta\in [1,p^*/q)$ \textit{such that} $t M(t)\le \theta\mathscr M(t)$ \textit{for any} $t\in[0,\infty)$,
\textit{where} $\mathscr M(t)=\displaystyle\int_0^t M(\tau)d\tau$ \textit{and} $p^*=Np/(N-p)$;
\item[$(M_2)$]
\textit{for any} $\tau>0$ \textit{there exists} $\kappa=\kappa(\tau)>0$ \textit{such that} $M(t)\geq \kappa$
\textit{for any} $t\geq \tau$.
\end{enumerate}
While $f:\Omega\times\mathbb R\to\mathbb R$ is a {\em Carath\'{e}odory} function verifying:
\begin{enumerate}
\item[$(f_1)$] {\em there exists an exponent $r\in (q\theta,p^*)$ such that  for any $\varepsilon>0$ there exists $\delta_\varepsilon=\delta(\varepsilon)>0$ and
$$
|f(x,t)|\leq q\theta\varepsilon\left|t\right|^{q\theta-1} +r\delta_\varepsilon\left|t\right|^{r-1}
$$
holds for a.e. $x\in \Omega$ and any $t\in\mathbb{R}$};
\item[$(f_2)$]
{\em there exist $\sigma\in (q\theta,p^*)$, $c>0$ and $t_0\geq0$ such that}
$$c\leq\sigma F(x,t)\leq tf(x,t)$$
{\em for a.e. $x\in\Omega$ and any }$|t|\geq t_0$, {\em where} $F(x,t)=\displaystyle\int^{t}_{0}f(x,\tau)d\tau$;
\item[$(f_3)$]
{\em $f(x,-t)=-f(x,t)$ for a.e. $x\in\Omega$ and any }$t\in\mathbb R$.
\end{enumerate}
Thus, we are ready to introduce our first result for \eqref{P}.

\begin{theorem}\label{T1.1}
Let $(M_1)-(M_2)$ and $(f_1)-(f_2)$ hold true.
Then, problem \eqref{P} admits a non-trivial weak solution.
\end{theorem}

The proof of Theorem \ref{T1.1} is based on the application of the classical mountain pass theorem. While, assuming the simmetric assumption in $(f_3)$, thanks to the Fountain theorem we are able to get the following multiplicity result for \eqref{P}. 
For this, we can replace assumption $(f_1)$ with
\begin{itemize}
\item[$(f_1')$] {\em there exists an exponent $r\in (p,p^*)$ and $C>0$ such that
$$
|f(x,t)|\leq C\left(1+|t|^{r-1}\right)
$$
holds for a.e. $x\in \Omega$ and any $t\in\mathbb{R}$}.
\end{itemize}
Hence, we obtain the following result.

\begin{theorem}\label{T2.2}
Let $(M_1)-(M_2)$ and $(f'_1)$, $(f_2)-(f_3)$ hold true.
Then, problem \eqref{P} has infinitely many weak solutions $\{u_j\}_j$ with unbounded energy.
\end{theorem}

In the second part of the paper, we consider the problem
\begin{equation}\label{P2}
\left\{\begin{array}{ll}
-M\left(\displaystyle\int_\Omega|\nabla u|^p dx\right)\Delta_p u-M\left(\displaystyle\int_\Omega a(x)|\nabla u|^q dx\right)
\mbox{div}\left(a(x)|\nabla u|^{q-2}\nabla u\right)=f(x,u) & \mbox{in } \Omega,\\
u=0 & \mbox{in } \partial\Omega.
\end{array}
\right.
\end{equation}
Even if the double phase operator does not explicitly appear in \eqref{P2}, this problem has still a variational structure set in the same double phase framework of \eqref{P}, as discussed in Section \ref{sec2}. However, because of the presence of two Kirchhoff coefficients, the study of problem \eqref{P2} is more delicate than \eqref{P}. In particular, in \eqref{P2} we must regard that the Lebesgue space $L_a^q(\Omega)$ with weight $a(\cdot)$ is a seminormed space, since $a(\cdot)$ could verify
\begin{equation}
\left|\left\{x\in\Omega:\,\,a(x)=0\right\}\right|>0,
\end{equation}
where $|\cdot|$ denotes the Lebesgue measure. 
Also, we observe that even when $M$ coincides with the Kirchhoff model $M(t)=m_1+m_2t^{\theta-1}$ for any $t\in[0,\infty)$, with $m_1\geq0$, $m_2>0$ two constants and $\theta$ given in $(M_1)$, problems \eqref{P} and \eqref{P2} are different.

We are now ready to provide the existence and multiplicity results for \eqref{P2}.
%considering the Euler-Lagrange functional $I$ of \eqref{P2}, given in Section \ref{sec2}.

\begin{theorem}\label{T3.3}
Let $(M_1)-(M_2)$ and $(f_1)-(f_2)$ hold true. Then, problem \eqref{P2} admits a non-trivial weak solution.
\end{theorem}

\begin{theorem}\label{T4.4}
Let $(M_1)-(M_2)$ and $(f'_1)$, $(f_2)-(f_3)$ hold true.
Then, problem \eqref{P2} has infinitely many weak solutions $\{u_j\}_j$ with unbounded energy.
\end{theorem}

The paper is organized as follows.
In Section \ref{sec2}, we introduce the basic properties of the Musielak-Orlicz and Musielak-Orlicz-Sobolev spaces and we set the variational structure of problems \eqref{P} and \eqref{P2}. In Section \ref{sec3},
we prove Theorems \ref{T1.1} and \ref{T2.2}. While, in Section \ref{sec4},
we prove Theorems \ref{T3.3} and \ref{T4.4}.

\section{Preliminaries}\label{sec2}

The function $\mathcal H:\Omega\times[0,\infty)\to[0,\infty)$ defined as
$$\mathcal H(x,t):=t^p+a(x)t^q,\quad\mbox{for a.e. }x\in\Omega\mbox{ and for any }t\in[0,\infty),
$$
with $1<p<q$ and $0\leq a \in L^1(\Omega)$, is a generalized N-function (N stands for {\em nice}), according to the definition in \cite{D,M}, and satisfies the so called $(\Delta_2)$ condition, that is
$$\mathcal H(x,2t)\leq t^q\mathcal H(x,t),\quad\mbox{for a.e. }x\in\Omega\mbox{ and for any }t\in[0,\infty).
$$
Therefore, by \cite{M} we can define the Musielak-Orlicz space $L^{\mathcal H}(\Omega)$ as
$$L^{\mathcal H}(\Omega):=\left\{u:\Omega\to\mathbb R\mbox{ measurable}:\,\,\varrho_{\mathcal H}(u)<\infty\right\},
$$
endowed with the Luxemburg norm
$$\|u\|_{\mathcal H}:=\inf\left\{\lambda>0:\,\,\varrho_{\mathcal H}\left(\frac{u}{\lambda}\right)\leq1\right\},$$
where $\varrho_{\mathcal H}$ denotes the $\mathcal H$-modular function, set as
\begin{equation}\label{rhoh}
\varrho_{\mathcal H}(u):=\int_\Omega\mathcal H(x,|u|)dx=\int_\Omega\left(|u|^p+a(x)|u|^q\right)dx.
\end{equation}
By \cite{CS,D}, the space $L^{\mathcal H}(\Omega)$ is a separable, uniformly convex, Banach space.
Furthermore, we define the weighted space
$$L^q_a(\Omega):=\left\{u:\Omega\to\mathbb R\mbox{ measurable}:\,\,\int_\Omega a(x)|u|^q dx<\infty\right\},
$$  
equipped with the seminorm
$$
\|u\|_{q,a}:=\left(\int_\Omega a(x)|u|^q dx\right)^{1/q}.
$$
By \cite[Proposition 2.15(i),(iv),(v)]{CS} we have the continuous embedding
$$L^q(\Omega)\hookrightarrow L^{\mathcal H}(\Omega)\hookrightarrow L^p(\Omega)\cap L_a^q(\Omega).
$$
While, by \cite[Proposition 2.1]{LD} we have the following relation between the norm $\|\cdot\|_{\mathcal H}$ and the $\mathcal H$-modular.

\begin{proposition}\label{P2.1}
Assume that $u\in L^{\mathcal H}(\Omega)$, $\{u_j\}_j\subset
L^{\mathcal H}(\Omega)$ and $c>0$. Then
\begin{itemize}
\item[$(i)$] for $u\neq0$, $\|u\|_{\mathcal H}=c\Leftrightarrow\varrho_{\mathcal H}\left(\frac{u}{c}\right)=1$;
\vspace{0.1cm}
\item[$(ii)$] $\|u\|_{\mathcal H}<1$ $(resp.=1,\,>1)$ $\Leftrightarrow \varrho_{\mathcal H}(u)<1$ $(resp.=1,\,>1)$;
\vspace{0.1cm}
\item[$(iii)$] $\|u\|_{\mathcal H}<1\Rightarrow\|u\|_{\mathcal H}^q\leq\varrho_{\mathcal H}(u)\leq\|u\|_{\mathcal H}^p$;
\vspace{0.1cm}
\item[$(iv)$] $\|u\|_{\mathcal H}>1\Rightarrow\|u\|_{\mathcal H}^p\leq\varrho_{\mathcal H}(u)\leq\|u\|_{\mathcal H}^q$;
\vspace{0.1cm}
\item[$(v)$] $\lim\limits_{j\rightarrow\infty}\|u_j\|_{\mathcal H}=0\,(\infty)\Leftrightarrow\lim\limits_{j\rightarrow\infty}\varrho_{\mathcal H}(u_j)=0\,(\infty)$.
\end{itemize}
\end{proposition}

The related Sobolev space $W^{1,\mathcal H}(\Omega)$ is defined by
$$W^{1,\mathcal H}(\Omega):=\left\{u\in L^{\mathcal H}(\Omega):\,\,|\nabla u|\in L^{\mathcal H}(\Omega)\right\},
$$
endowed with the norm
\begin{equation}\label{norma}
\|u\|_{1,\mathcal H}:=\|u\|_{\mathcal H}+\|\nabla u\|_{\mathcal H},
\end{equation}
where we write $\|\nabla u\|_{\mathcal H}=\||\nabla u|\|_{\mathcal H}$ to simplify the notation.
We denote by $W^{1,\mathcal H}_0(\Omega)$ the completion of $C^\infty_0(\Omega)$ in $W^{1,\mathcal H}(\Omega)$ which can be endowed with the norm
$$\|u\|:=\|\nabla u\|_{\mathcal H},
$$
equivalent to the norm set in \eqref{norma}, thanks to \cite[Proposition 2.18(iv)]{CS} whenever \eqref{cruciale} holds true. Also, by \cite[Proposition 2.15(ii)-(iii)]{CS} we have the following embeddings.

\begin{proposition}\label{P2.3}
For any $\nu\in[1,p^*]$ there exists a constant $C_\nu=C(N,p,q,\nu,\Omega)>0$ such that
$$\|u\|_\nu^\nu\leq C_\nu\|u\|^\nu
$$
for any $u\in W^{1,\mathcal H}_0(\Omega)$.
Moreover, the embedding $W^{1,\mathcal H}_0(\Omega)\hookrightarrow L^\nu(\Omega)$ is compact for any $\nu\in[1,p^*)$.
\end{proposition}

Let us define the operator $L:W^{1,\mathcal H}_0(\Omega)\to\left(W^{1,\mathcal H}_0(\Omega)\right)^*$ such that
$$\langle L(u),v\rangle:=\int_\Omega\left(|\nabla u|^{p-2}+a(x)|\nabla u|^{q-2}\right)\nabla u\cdot\nabla v dx,
$$
for any $u$, $v\in W^{1,\mathcal H}_0(\Omega)$.
Here, $\left(W^{1,\mathcal H}_0(\Omega)\right)^*$ denotes the dual space of $W^{1,\mathcal H}_0(\Omega)$ and $\langle \cdot,\cdot\rangle$ is the related dual pairing. Then, we have the following crucial result, given in \cite[Proposition 3.1(ii)]{LD}.

\begin{proposition}\label{P2.4}
$L:W^{1,\mathcal H}_0(\Omega)\to\left(W^{1,\mathcal H}_0(\Omega)\right)^*$ is a mapping of $(S_+)$ type, that is if $u_j\rightharpoonup u$ in $W^{1,\mathcal H}_0(\Omega)$ and $\limsup\limits_{j\to\infty}\langle L(u_j)-L(u),u_j-u\rangle\leq0$, then $u_j\to u$ in $W^{1,\mathcal H}_0(\Omega)$.
\end{proposition}

We are now ready to introduce the variational setting for problems \eqref{P} and \eqref{P2}.
We say that a function $u\in W^{1,\mathcal H}_0(\Omega)$  is a weak solution of \eqref{P} if
$$
M[\phi_{\mathcal H}(\nabla u)]\langle L(u),\varphi\rangle
=\int_\Omega f(x,u)\varphi dx,
$$
for any $\varphi\in W^{1,\mathcal H}_0(\Omega)$, where we denote
$$\phi_{\mathcal H}(u):=\int_\Omega\left(\frac{|u|^p}{p}+a(x)\frac{|u|^q}{q}\right)dx.
$$
Clearly, the weak solutions of \eqref{P} are exactly the critical points of the Euler-Lagrange functional $J:W^{1,\mathcal H}_0(\Omega)\to\mathbb R$, given by
$$J(u):=\mathscr M[\phi_{\mathcal H}(\nabla u)]-\int_\Omega F(x,u)dx,
$$
which is well defined and of class $C^{1}$ on $W^{1,\mathcal H}_0(\Omega)$. 

Similarly, a function $u\in W^{1,\mathcal H}_0(\Omega)$  is a weak solution of \eqref{P2} if
$$
M(\|\nabla u\|_p^p)\int_\Omega |\nabla u|^{p-2}\nabla u\cdot\nabla\varphi dx+
M(\|\nabla u\|_{q,a}^q)\int_\Omega a(x)|\nabla u|^{q-2}\nabla u\cdot\nabla\varphi dx
=\int_\Omega f(x,u)\varphi dx,
$$
for any $\varphi\in W^{1,\mathcal H}_0(\Omega)$. In this case, the Euler-Lagrange functional $I:W^{1,\mathcal H}_0(\Omega)\to\mathbb R$ associated to \eqref{P2} is set as
$$I(u):=\frac{1}{p}\mathscr M(\|\nabla u\|_p^p)+\frac{1}{q}\mathscr M(\|\nabla u\|_{q,a}^q)-\int_\Omega F(x,u)dx,
$$
which is well defined and of class $C^{1}$ on $W^{1,\mathcal H}_0(\Omega)$, thanks to Proposition \ref{P2.1} and \eqref{rhoh}.

\section{Proof of Theorems \ref{T1.1} and \ref{T2.2}}\label{sec3}

We start the section verifying that functional $J$ satisfies the geometric features of the mountain pass theorem, see e.g. \cite[Theorem 1.15]{Wil}.

\begin{lemma}\label{L3.1}
Let $(M_1)-(M_2)$ and $(f_1)$ hold true. Then, there exist $\rho\in(0,1]$ and $\alpha=\alpha(\rho)>0$ such that $J(u)\geq \alpha$ for any $u\in W^{1,\mathcal H}_0(\Omega)$, with $\|u\|=\rho$.
\end{lemma}

\begin{proof}
By $(f_1)$, for any $\varepsilon>0$ we have a $\delta_\varepsilon>0$ such that
\begin{equation}\label{3.1}
|F(x,t)|\leq\varepsilon|t|^{q\theta}+\delta_\varepsilon|t|^r,\quad\mbox{for a.e. }x\in\Omega\mbox{ and any }t\in\mathbb R.
\end{equation}
While, by integrating $(M_1)$ and considering that $M(t)>0$ for any $t>0$ by $(M_2)$, we have
\begin{equation}\label{3.2}
\mathscr M(t)\geq\mathscr M(1)t^{\theta}\quad\mbox{for any }t\in[0,1].
\end{equation}
Also, by Proposition \ref{P2.1} for any $u\in W^{1,\mathcal H}_0(\Omega)$ with $\|u\|\leq1$, we get
\begin{equation}\label{3.3}
\phi_{\mathcal H}(\nabla u)\leq\frac{1}{p}\varrho_{\mathcal H}(\nabla u)\leq\frac{1}{p}\|u\|^p<1,
\end{equation}
being $1<p<q$.
Thus, by \eqref{3.1}-\eqref{3.3} and Propositions \ref{P2.1}-\ref{P2.3}, for any  $u\in W^{1,\mathcal H}_0(\Omega)$ with $\|u\|\leq1$, we obtain
\begin{align*}
J(u)&\geq\mathscr M(1)[\phi_{\mathcal H}(\nabla u)]^\theta-\varepsilon\|u\|_{q\theta}^{q\theta}-\delta_\varepsilon\|u\|_r^r
\geq\frac{\mathscr M(1)}{q^\theta}[\varrho_{\mathcal H}(\nabla u)]^\theta
-\varepsilon C_{q\theta}\|u\|^{q\theta}-\delta_\varepsilon C_r\|u\|^r\\
&\geq\left(\frac{\mathscr M(1)}{q^\theta}-\varepsilon C_{q\theta}\right)\|u\|^{q\theta}-\delta_\varepsilon C_r\|u\|^r.
\end{align*}
Therefore, choosing $\varepsilon>0$ sufficiently small so that
$$\mu_\varepsilon:=\frac{\mathscr M(1)}{q^\theta}-\varepsilon C_{q\theta}>0,
$$ 
for any $u\in W^{1,\mathcal H}_0(\Omega)$ with 
$\|u\|=\rho\in\big(0,\min\{1,[\mu_\varepsilon/(2\delta_\varepsilon C_r)]^{1/(r-q\theta)}\}\big]$, we get
$$
J(u)\geq\left(\mu_\varepsilon-\delta_\varepsilon C_r\rho^{r-q\theta}\right)\rho^{q\theta}:=\alpha>0,
$$
concluding the proof.
\end{proof}

\begin{lemma}\label{L3.2}
Let $(M_1)-(M_2)$ and $(f_1)-(f_2)$ hold true. Then, there exists $e\in W^{1,\mathcal H}_0(\Omega)$ such that $J(e)<0$ and $\|e\|>1$.
\end{lemma}
\begin{proof}
By $(f_1)$ and $(f_2)$, there exist $d_1>0$ and $d_2\geq0$ such that
\begin{equation}\label{3.4}
F(x,t)\geq d_1|t|^\sigma-d_2\quad\mbox{for a.e. }x\in\Omega\mbox{ and any }t\in\mathbb R.
\end{equation}
By integrating $(M_1)$, we have
\begin{equation}\label{3.5}
\mathscr M(t)\leq \mathscr M(1)t^\theta\quad\mbox{for any }t\geq1.
\end{equation}
While, by Proposition \ref{P2.1} for any $u\in W^{1,\mathcal H}_0(\Omega)$ with $\|u\|\geq q^{1/p}>1$, we get
\begin{equation}\label{3.6}
\phi_{\mathcal H}(\nabla u)\geq\frac{1}{q}\varrho_{\mathcal H}(\nabla u)\geq\frac{1}{q}\|u\|^p\geq1.
\end{equation}
Thus, if $\varphi\in W^{1,\mathcal H}_0(\Omega)$ with $\|\varphi\|=1$, then by \eqref{3.4}-\eqref{3.6} for any $t\geq q^{1/p}$ we have
$$
J(t\varphi)\leq \mathscr M(1)t^{q\theta}\left[\phi_{\mathcal H}(\nabla\varphi)\right]^{\theta}-t^\sigma d_1\|\varphi\|_\sigma^\sigma-d_2|\Omega|.
$$
Since $\sigma>q\theta$ by $(f_2)$, passing to the limit as $t\rightarrow\infty$ we get $J(t\varphi)\to-\infty$. Thus, the assertion follows by taking $e=t_{\infty}\varphi$, with $t_{\infty}$ sufficiently large.
\end{proof}

We recall that a functional $\mathcal F:W^{1,\mathcal H}_0(\Omega)\to\mathbb R$ fulfills the Palais-Smale condition $(PS)$ if any sequence $\{u_j\}_j\subset W^{1,\mathcal H}_0(\Omega)$ satisfying
\begin{equation}\label{3.7}
\{\mathcal F(u_j)\}_j\mbox{ is bounded and }\mathcal F'(u_j)\to 0\mbox{ in }\left(W^{1,\mathcal H}_0(\Omega)\right)^{*}\mbox{ as }j\rightarrow\infty,
\end{equation}
admits a convergent subsequence in $W^{1,\mathcal H}_0(\Omega)$. Now, we are going to verify the $(PS)$ condition for $J$.

\begin{lemma}\label{L3.3}
Let $(M_1)-(M_2)$ and $(f_1)-(f_2)$ hold true. Then, the functional $J$ verifies the $(PS)$ condition.
\end{lemma}
\begin{proof}
Let $\{u_j\}_j\subset W^{1,\mathcal H}_0(\Omega)$ be a sequence satisfying \eqref{3.7} with $\mathcal F=J$. 

We first show that $\{u_j\}_j$ is bounded in $W^{1,\mathcal H}_0(\Omega)$, arguing by contradiction. Then, there exists a subsequence, still denoted by $\{u_j\}_j$  and $n\in\mathbb{N}$ such that $\lim\limits_{j\to\infty}\|u_j\|=\infty$ and $\|u_j\|\geq q^{1/p}$ for any $j\geq n$. By $(M_2)$ with $\tau=1$, there exists $\kappa=\kappa(1)>0$ such that, thanks to \eqref{3.6}, we have
\begin{equation}\label{3.8}
M\left[\phi_{\mathcal H}(\nabla u_j)\right]\geq\kappa\text{ for any }j\geq n.
\end{equation}
Thus, according to $(M_1)$, $(f_2)$ and \eqref{3.8}, we get
\begin{align}\label{3.9}
J(u_j)-\frac{1}{\sigma}\langle J'(u_j), u_j\rangle
=&\mathscr M\left[\phi_{\mathcal H}(\nabla u_j)\right]
-\frac{1}{\sigma}M\left[\phi_{\mathcal H}(\nabla u_j)\right]\varrho_{\mathcal H}(\nabla u_j)
-\int_{\Omega}\left[F(x,u_j)-\frac{1}{\sigma}f(x,u_j)u_j\right]dx\nonumber\\
\geq&\left(\frac{1}{\theta}-\frac{q}{\sigma}\right)M\left[\phi_{\mathcal H}(\nabla u_j)\right]\phi_{\mathcal H}(\nabla u_j)
-\int_{\Omega_j}\left[F(x,u_j)-\frac{1}{\sigma}f(x,u_j)u_j\right]^+dx\\
\geq&\left(\frac{1}{\theta}-\frac{q}{\sigma}\right)M\left[\phi_{\mathcal H}(\nabla u_j)\right]\frac{\varrho_{\mathcal H}(\nabla u_j)}{q}
-D,\nonumber
\end{align}
since $\sigma>q\theta$ by $(f_2)$, where
\begin{equation}\label{omegajay}
\Omega_j:=\left\{x\in\Omega:\,\,|u_j(x)|\leq t_0\right\}\quad\mbox{ and }\quad
D:=|\Omega|\sup_{x\in\Omega,\,|t|\leq t_0}\left[F(x,t)-\frac{1}{\sigma}f(x,t)t\right]^+<\infty,
\end{equation}
with the last inequality is consequence of $(f_1)$ and $t^+=\max\{t,0\}$ denotes the positive part of a number $t\in\mathbb R$.
Thus, by \eqref{3.7} there exist $c_1$, $c_2>0$ such that \eqref{3.8}-\eqref{3.9} and Proposition \ref{P2.1} yield at once that as $j\rightarrow\infty$,
$$
c_1+c_2\|u_j\|+o(1)\geq\left(\frac{1}{\theta}-\frac{q}{\sigma}\right)\frac{\kappa}{q}\|u_j\|^p-D
$$
giving the desired contradiction, since $p>1$.

Hence, $\{u_j\}_j$ is bounded in $W^{1,\mathcal H}_0(\Omega)$.
By Propositions \ref{P2.1}-\ref{P2.3}, the reflexivity of $W^{1,\mathcal H}_0(\Omega)$ and \cite[Theorem 4.9]{B}, there exists a subsequence, still denoted by $\{u_j\}_j$, and $u\in W^{1,\mathcal H}_0(\Omega)$ such that
\begin{equation}\label{3.10}
\begin{gathered}
u_j\rightharpoonup u\mbox{ in }W^{1,\mathcal H}_0(\Omega),\qquad \nabla u_j\rightharpoonup\nabla u\mbox{ in }\left[L^{\mathcal H}(\Omega)\right]^N, \qquad\phi_{\mathcal H}(\nabla u_j)\rightarrow\ell,\\
u_j\to u\mbox{ in }L^\nu(\Omega),\qquad u_j(x)\rightarrow u(x)\mbox{ a.e. in }\Omega,
\end{gathered}
\end{equation}
as $j\to\infty$, with $\nu\in[1,p^*)$.
Of course, if $\ell=0$ then, since $\phi_{\mathcal H}(v)\geq\rho_{\mathcal H}(v)/q\geq0$ for any $v\in W^{1,\mathcal H}_0(\Omega)$, by Proposition \ref{P2.1} we have $u_j\to0$ in $W^{1,\mathcal H}_0(\Omega)$. Hence, let us suppose $\ell>0$.

By $(f_1)$ with $\varepsilon=1$, the H\"older inequality, the boundedness of $\{u_j\}_j$, \eqref{3.10} applied with $\nu=r$ and $\nu=q\theta$ thanks to $(M_1)$, we obtain
\begin{equation}\label{3.11}
\begin{aligned}
\left|\int_\Omega f(x,u_j)(u_j-u)dx\right|
&\leq\int_\Omega\left(q\theta|u_j|^{q\theta-1}+r\delta_1|u_j|^{r-1}\right)|u_j-u|dx\\
&\leq C\left(\|u_j-u\|_{q\theta}+\|u_j-u\|_r\right)\to 0
\end{aligned}
\end{equation}
as $j\to\infty$, for a suitable $C>0$. Thus, by \eqref{3.7}, \eqref{3.10} and \eqref{3.11}, we get
\begin{equation}\label{3.12}
\begin{aligned}
o(1)=\langle J'(u_j),u_j-u\rangle&=M\left[\phi_{\mathcal H}(\nabla u_j)\right]\langle L(u_j),u_j-u\rangle
-\int_\Omega f(x,u_j)(u_j-u)dx\\
&=M(\ell)\langle L(u_j),u_j-u\rangle+o(1)
\end{aligned}
\end{equation}
as $j\to\infty$.
By the H\"older inequality, \eqref{rhoh} and Proposition \ref{P2.1}, we see that functional
$$
G:g\in\left[L^{\mathcal H}(\Omega)\right]^N\mapsto\int_\Omega\left(|\nabla u|^{p-2}\nabla u+a(x)|\nabla u|^{q-2}\nabla u\right)\cdot g\,dx
$$
is linear and bounded.
Hence, by \eqref{3.10} we have
\begin{equation}\label{3.13}
\langle L(u),u_j-u\rangle
=\int_\Omega\left(|\nabla u|^{p-2}\nabla u+a(x)|\nabla u|^{q-2}\nabla u\right)\cdot(\nabla u_j-\nabla u)dx\to 0\mbox{ as }j\to\infty.
\end{equation}
Thus, combining \eqref{3.12}-\eqref{3.13} and Proposition \ref{P2.4}, since $M(\ell)>0$ by $(M_2)$, we conclude that $u_j\to u$ in $W^{1,\mathcal H}_0(\Omega)$. This completes the proof.
\end{proof}

\begin{remark}\label{maq}
The same result in Lemma \ref{L3.3} holds assuming $(f_1')$ instead of $(f_1)$.
\end{remark}

We are now ready to prove Theorem \ref{T1.1}.

\begin{proof}[Proof of Theorem \ref{T1.1}]
Since $J(0)=0$, by Lemmas \ref{L3.1}-\ref{3.3} and the mountain pass theorem, the existence of a nontrivial weak solution of \eqref{P} follows at once.
\end{proof}

In order to verify Theorem \ref{T2.2}, we use the Fountain theorem given in \cite[Theorem 3.6]{Wil} applied to the functional $J$. For this, we first need some notations. Since $W^{1,\mathcal H}_0(\Omega)$ is a reflexive and separable Banach space, there are two sequences $\{e_j\}_j\subset W^{1,\mathcal H}_0(\Omega)$ and $\{e_j^*\}_j\subset \left(W^{1,\mathcal H}_0(\Omega)\right)^*$ such that
$$
W^{1,\mathcal H}_0(\Omega)=\overline{\mathrm{span}\{e_j:\, j\in \mathbb{N}\}},\qquad  (W^{1,\mathcal H}_0(\Omega))^*=\overline{\mathrm{span}\{e_j^*:\, j\in \mathbb{N}\}}
$$
and 
\[\langle e_i^*, e_j\rangle=\begin{cases}
  1, &i=j, \\
  0, & i\neq j.
\end{cases}
\]
Then, for any $j\in\mathbb N$, we can set
\begin{equation}\label{3.15}
X_j:=\mathrm{span}\{e_j\},\qquad Y_i:=\bigoplus_{i=1}^j X_i,\qquad Z_j:=\bigoplus_{i=j}^{\infty} X_i,
\qquad\beta_j:=\sup_{u\in Z_j,\,\|u\|=1} \|u\|_r,
\end{equation}
where $r$ given in $(f_1)$.
The geometric structure of the Fountain theorem in \cite[Theorem 3.6]{Wil}, applied to an even functional $\mathcal F:W^{1,\mathcal H}_0(\Omega)\to\mathbb R$, requires to show that for any $j\in \mathbb{N}$ there exist $\rho_j>\gamma_j>0$ such that
\begin{align}\label{cond1}
&a_j:=\max_{u\in Y_j,\,\|u\|=\rho_j} \mathcal F(u)\leq 0\\
\label{cond2}
&b_j:=\inf_{u\in Z_j,\,\|u\|=\gamma_j} \mathcal F(u)\to \infty \mbox{ as } j\to \infty.
\end{align}
In order to verify \eqref{cond2}, we use an asymptotic property of $\beta_j$ proved in \cite[Lemma 7.1]{LD}.

\begin{proof}[Proof of Theorem \ref{T2.2}]
By $(f_3)$ we have that $J$ is an even functional. While, $J$ satisfies the $(PS)$ condition thanks to Remark \ref{maq}.
Thus, in order to apply the Fountain theorem in \cite[Theorem 3.6]{Wil}, for any $j\in \mathbb{N}$ we need to find $\rho_j>\gamma_j>0$ such that \eqref{cond1} and \eqref{cond2} hold true for $\mathcal F=J$.

Let us first prove \eqref{cond2}.
By $(f'_1)$ we have
\begin{equation}\label{3.1n}
|F(x,t)|\leq C\left(|t|+|t|^r\right),\quad\mbox{for a.e. }x\in\Omega\mbox{ and any }t\in\mathbb R,
\end{equation}
with a possibly new $C>0$.
By $(M_2)$ with $\tau=1$, there exists $\kappa=\kappa(1)>0$ such that, thanks to Proposition \ref{P2.1} and \eqref{3.6}, we have
\begin{equation}\label{3.18}
M\left[\phi_{\mathcal H}(\nabla u)\right]\geq\kappa,
\end{equation}
for any $u\in Z_j$ with $\|u\|\geq q^{1/p}$. Thus, by $(M_1)$, \eqref{3.6}, \eqref{3.1n}, \eqref{3.18}, the H\"older inequality, the definition of $\beta_j$ in \eqref{3.15} and the fact that $r>p$, for any $u\in Z_j$ with $\|u\|\geq q^{1/p}>1$ we obtain
\begin{equation}\label{3.19}
\begin{aligned}
J(u)&\geq\frac{1}{\theta}M[\phi_{\mathcal H}(\nabla u)]\phi_{\mathcal H}(\nabla u)-C\|u\|_{1}-C\|u\|_r^r
\geq\frac{\kappa}{q\theta}\|u\|^p
-C|\Omega|^{(r-1)/r}\|u\|_r-C\|u\|_r^r\\
&\geq\frac{\kappa}{q\theta}\|u\|^p\!-\!\beta_jC|\Omega|^{(r-1)/r}\|u\|
\!-\!\beta_j^rC\|u\|^r
\!\geq\!\left[\frac{\kappa}{q\theta}-C\left(\beta_j|\Omega|^{(r-1)/r}+\beta_j^r\right)\|u\|^{r-p}\right]
\|u\|^p.
\end{aligned}
\end{equation}
Now, let us choose
$$\gamma_j:=\left[\frac{\kappa}{2q\theta}\cdot
\frac{1}{C\left(\beta_j|\Omega|^{(r-1)/r}+\beta_j^r\right)}\right]^{1/(r-p)}
$$
such that $\gamma_j\to\infty$ as $j\to\infty$, since $\beta_j\to 0$ as $j\to\infty$ by \cite[Lemma 7.1]{LD} and $r>p$ by $(f'_1)$.
Then, by \eqref{3.19}, for any $u\in Z_j$ with $\|u\|=\gamma_j$ we get
$$
J(u)\geq\frac{\kappa}{2q\theta}\gamma_j^p\to\infty\mbox{ as }j\to\infty,
$$
which gives the validity of condition \eqref{cond2}. 

In order to prove \eqref{cond1} let us fix $j\in\mathbb N$. Since the norms are topological equivalent in $Y_j$, there exists $c(j)>0$ such that
\begin{equation}\label{3.17}
\|u\|^\sigma\leq c(j)\|u\|_\sigma^\sigma,
\end{equation}
for any $u\in Y_j$.
Also, by Proposition \ref{P2.1} for any $u\in Y_j$ with $\|u\|\geq1$ we get
$$
\phi_{\mathcal H}(\nabla u)\leq\frac{1}{p}\varrho_{\mathcal H}(\nabla u)\leq\frac{1}{p}\|u\|^q,
$$
being $1<p<q$.
From this, by \eqref{3.4}-\eqref{3.6} and \eqref{3.17}, for any $u\in Y_j$ with $\|u\|\geq q^{1/p}$ we have
$$
J(u)\leq \mathscr M(1)\left[\phi_{\mathcal H}(\nabla u)\right]^{\theta}-d_1\|u\|_\sigma^\sigma-d_2|\Omega|
\leq\frac{\mathscr M(1)}{p^{\theta}}\|u\|^{q\theta}-d_1c(j)\|u\|^\sigma-d_2|\Omega|,
$$
which yields \eqref{cond1} with $\rho_j>\max\{q^{1/p},\gamma_j\}$ sufficiently large, since $\sigma>q\theta$ by $(f_2)$.

Thus, we can apply \cite[Theorem 3.6]{Wil} to functional $J$ and we get an unbounded sequence of critical points of $J$ with unbounded energy, concluding the proof of Theorem \ref{T2.2}.

\end{proof}

\section{Proof of Theorems \ref{T3.3} and \ref{T4.4}}\label{sec4}
As in Theorem \ref{T1.1}, we apply the mountain pass theorem to prove Theorem \ref{T3.3}, starting from the geometry of $I$.

\begin{lemma}\label{L4.1}
Let $(M_1)-(M_2)$ and $(f_1)$ hold true. 
Then, there exist $\rho\in(0,1]$ and $\alpha=\alpha(\rho)>0$ such that $I(u)\geq \alpha$ for any $u\in W^{1,\mathcal H}_0(\Omega)$, with $\|u\|=\rho$.
\end{lemma}

\begin{proof}
Let us first consider $u\in W^{1,\mathcal H}_0(\Omega)$ with $\|u\|\leq1$.
By Proposition \ref{P2.1} and \eqref{rhoh}, also $\|\nabla u\|_p\leq1$ and $\|\nabla u\|_{q,a}\leq1$.
Thus, by \eqref{3.1}, \eqref{3.2}, Propositions \ref{P2.1}-\ref{P2.3} and the Jensen inequality we have
\begin{align*}
I(u)&\geq\frac{\mathscr M(1)}{p}\|\nabla u\|_p^{p\theta}+\frac{\mathscr M(1)}{q}\|\nabla u\|_{q,a}^{q\theta}-\varepsilon\|u\|_{q\theta}^{q\theta}-\delta_\varepsilon\|u\|_r^r\\
&\geq\frac{\mathscr M(1)}{q2^{\theta-1}}[\varrho_{\mathcal H}(\nabla u)]^\theta
-\varepsilon C_{q\theta}\|u\|^{q\theta}-\delta_\varepsilon C_r\|u\|^r\\
&\geq\left(\frac{\mathscr M(1)}{q2^{\theta-1}}-\varepsilon C_{q\theta}\right)\|u\|^{q\theta}-\delta_\varepsilon C_r\|u\|^r.
\end{align*}
Therefore, choosing $\varepsilon>0$ sufficiently small so that
$$\mu_\varepsilon:=\frac{\mathscr M(1)}{q2^{\theta-1}}-\varepsilon C_{q\theta}>0,
$$ 
for any $u\in W^{1,\mathcal H}_0(\Omega)$ with 
$\|u\|=\rho\in\big(0,\min\{1,1/K_p^p,1/K_q^q,[\mu_\varepsilon/(2\delta_\varepsilon C_r)]^{1/(r-q\theta)}\}\big]$, we obtain
$$
I(u)\geq\left(\mu_\varepsilon-\delta_\varepsilon C_r\rho^{r-q\theta}\right)\rho^{q\theta}:=\alpha>0,
$$
concluding the proof.
\end{proof}

\begin{lemma}\label{L4.2}
Let $(M_1)-(M_2)$ and $(f_1)-(f_2)$ hold true.
Then, there exists $e\in W^{1,\mathcal H}_0(\Omega)$ such that $I(e)<0$, $\|\nabla e\|_p\geq1$ and $\|\nabla e\|_{q,a}\geq1$.
\end{lemma}
\begin{proof}
If $\varphi\in W^{1,\mathcal H}_0(\Omega)$ with $\|\nabla\varphi\|_p\geq1$ and $\|\nabla\varphi\|_{q,a}\geq1$, then by \eqref{3.4}-\eqref{3.5} for any $t\geq 1$ we have
$$
I(t\varphi)\leq \frac{\mathscr M(1)}{p}t^{p\theta}\|\nabla\varphi\|_p^{p\theta}
+\frac{\mathscr M(1)}{q}t^{q\theta}\|\nabla\varphi\|_{q,a}^{q\theta}
-t^\sigma d_1\|\varphi\|_\sigma^\sigma-d_2|\Omega|.
$$
Since $\sigma>q\theta>p\theta$ by $(f_2)$, passing to the limit as $t\rightarrow\infty$ we get $I(t\varphi)\to-\infty$. Thus, the assertion follows by taking $e=t_{\infty}\varphi$, with $t_{\infty}$ sufficiently large.
\end{proof}

The verification of the $(PS)$ condition for $I$ is fairly delicate. Indeed, in the functional $I$ we must handle two Kirchhoff coefficients, with $M$ possibly degenerate, that is verifying $M(0)=0$.

\begin{lemma}\label{L4.3}
Let  $(M_1)-(M_2)$ and $(f_1)-(f_2)$ hold true. Then, the functional $I$ verifies the $(PS)$ condition.
\end{lemma}
\begin{proof}
Let $\{u_j\}_j\subset W^{1,\mathcal H}_0(\Omega)$ be a sequence satisfying \eqref{3.7} with $\mathcal F=I$. 

We first show that $\{u_j\}_j$ is bounded in $W^{1,\mathcal H}_0(\Omega)$, arguing by contradiction.
Then, going to a subsequence, still denoted by $\{u_j\}_j$, we have $\lim\limits_{j\to\infty}\|u_j\|=\infty$ and there exists $n\in\mathbb{N}$ such that $\|u_j\|\geq1$ for any $j\geq n$ and thanks to Proposition \ref{P2.1} we also have $\lim\limits_{j\to\infty}\varrho_{\mathcal H}(\nabla u_j)=\infty$. 
Hence, by \eqref{rhoh} either the sequence $\{|\nabla u_j|\}_j$ diverges both in~$L^p(\Omega)$ and in~$L^q_a(\Omega)$, or $\{|\nabla u_j|\}_j$ diverges in one space and it is bounded in the other. Suppose that the first case occurs. Then, up to going to another subsequence, we have
\begin{equation}\label{infty}
\lim_{j\to\infty}\|\nabla u_j\|_p=\infty,\qquad
\|\nabla u_j\|_p\ge1,\qquad
\lim_{j\to\infty}\|\nabla u_j\|_{q,a}=\infty,\qquad
\|\nabla u_j\|_{q,a}\ge1,
\end{equation}
for any $j\geq n$.
By $(M_2)$, with $\tau=1$, there exists $\kappa>0$ such that
\begin{equation}\label{mk}
M(\|\nabla u_j\|_p^p)\geq \kappa\quad\mbox{and}\quad
M(\|\nabla u_j\|_{q,a}^q)\geq \kappa,\quad\mbox{for any }j\geq n.
\end{equation}
Thus, by $(M_1)$, $(f_2)$ and \eqref{mk}, we get
\begin{align}\label{n3.9}
I(u_j)-\frac{1}{\sigma}\langle I'(u_j), u_j\rangle
\!=&\frac{1}{p}\mathscr M(\|\nabla u_j\|_p^p)\!+\!\frac{1}{q}\mathscr M(\|\nabla u_j\|_{q,a}^q)
\!-\!\frac{1}{\sigma}\left[M(\|\nabla u_j\|_p^p)\|\nabla u_j\|_p^p\!+\!
M(\|\nabla u_j\|_{q,a}^q)\|\nabla u_j\|_{q,a}^q\right]\nonumber\\
&-\int_{\Omega}\left[F(x,u_j)-\frac{1}{\sigma}f(x,u_j)u_j\right]dx\nonumber\\
\geq&\left(\frac{1}{p\theta}-\frac{1}{\sigma}\right)M(\|\nabla u_j\|_p^p)\|\nabla u_j\|_p^p
+\left(\frac{1}{q\theta}-\frac{1}{\sigma}\right)M(\|\nabla u_j\|_{q,a}^q)\|\nabla u_j\|_{q,a}^q\\
&-\int_{\Omega_j}\left[F(x,u_j)-\frac{1}{\sigma}f(x,u_j)u_j\right]^+dx\nonumber\\
\geq&\left(\frac{1}{q\theta}-\frac{1}{\sigma}\right)\kappa\varrho_{\mathcal H}(\nabla u_j)
-D,\nonumber
\end{align}
being $p\theta<q\theta<\sigma$ by $(f_2)$, with $\Omega_j$ and $D$ defined as in \eqref{omegajay}. Hence, by \eqref{3.7}
there exist  $c_1$, $c_2>0$ such that \eqref{n3.9} and Proposition \ref{P2.1} imply
\begin{equation}\label{4bis}
c_1+c_2\|u_j\|+o(1)\ge\left(\frac{1}{q\theta}-\frac{1}{\sigma}\right)\kappa\|u_j\|^p-D,
\end{equation}
as $j\to\infty$, giving the desired contradiction since $p>1$.

It remains to consider the latter case, that is when $\{|\nabla u_j|\}_j$ diverges in one space, but is bounded in the other. Suppose that going to a further subsequence
\begin{equation}\label{infty2}
\lim_{j\to\infty}\|\nabla u_j\|_p=\infty,\qquad
\|\nabla u_j\|_p\ge1,\qquad
\sup_{j\in\mathbb{N}}\|\nabla u_j\|_{q,a}<\infty,
\end{equation}
for any $j\geq n$. Arguing as in \eqref{n3.9} and \eqref{4bis}, we now obtain as $j\to\infty$
$$
c_1+c_2\|u_j\|+o(1)\ge\left(\frac{1}{p\theta}-\frac{1}{\sigma}\right)\kappa\|\nabla u_j\|_p^p-D
$$
which yields by \eqref{rhoh} and Proposition \ref{P2.1}
\begin{equation}\label{4.2}
0<\left(\frac{1}{p\theta}-\frac{1}{\sigma}\right)\kappa\le c_2
\frac{\left(\|\nabla u_j\|_p^p+\|\nabla u_j\|_{q,a}^q\right)^{1/p}}
{\|\nabla u_j\|_p^p}+o(1),
\end{equation}
as $j\to\infty$. Again \eqref{4.2} cannot occur by~\eqref{infty2}.
The claim is now completely proved.

Hence, $\{u_j\}_j$ is bounded in $W^{1,\mathcal H}_0(\Omega)$.
By Propositions \ref{P2.1}-\ref{P2.3}, the reflexivity of $W^{1,\mathcal H}_0(\Omega)$ and \cite[Theorem 4.9]{B}, there exists a subsequence, still denoted by $\{u_j\}_j$, and $u\in W^{1,\mathcal H}_0(\Omega)$ such that
\begin{equation}\label{n3.10}
\begin{gathered}
u_j\rightharpoonup u\mbox{ in }W^{1,\mathcal H}_0(\Omega),
\qquad\nabla u_j\rightharpoonup\nabla u\mbox{ in }\left[L^{\mathcal H}(\Omega)\right]^N,
\qquad\|\nabla u_j\|_p\rightarrow\ell_p,\\
\|\nabla u_j\|_{q,a}\rightarrow\ell_q,\qquad
u_j\to u\mbox{ in }L^\nu(\Omega),\qquad u_j(x)\rightarrow u(x)\mbox{ a.e. in }\Omega,
\end{gathered}
\end{equation}
as $j\to\infty$, with $\nu\in[1,p^*)$.
By \eqref{3.7}, \eqref{3.11} and \eqref{n3.10}, we have
\begin{align}\label{n3.12}
o(1)=\langle I'(u_j),u_j-u\rangle=&M(\|\nabla u_j\|_p^p)\int_\Omega|\nabla u_j|^{p-2}\nabla u_j\cdot(\nabla u_j-\nabla u)dx\nonumber\\
&+M(\|\nabla u_j\|_{q,a}^q)\int_\Omega a(x)|\nabla u_j|^{q-2}\nabla u_j\cdot(\nabla u_j-\nabla u)dx
-\int_\Omega f(x,u_j)(u_j-u)dx\nonumber\\
=&M(\ell_p^p)\int_\Omega|\nabla u_j|^{p-2}\nabla u_j\cdot(\nabla u_j-\nabla u)dx\\
&+M(\ell_q^q)\int_\Omega a(x)|\nabla u_j|^{q-2}\nabla u_j\cdot(\nabla u_j-\nabla u)dx+o(1)\nonumber
\end{align}
as $j\to\infty$. 
%\textcolor{red}{Io direi cosi: By definition we have
%\begin{align}
%\langle I'(u_j),u_j-u\rangle&=M(\|\nabla u_j\|_p^p)\int_\Omega|\nabla u_j|^{p-2}\nabla u_j\cdot(\nabla u_j-\nabla u)dx\\
%\nonumber
%&+M(\|\nabla u_j\|_{q,a}^q)\int_\Omega a(x)|\nabla u_j|^{q-2}\nabla u_j\cdot(\nabla u_j-\nabla u)dx
%-\int_\Omega f(x,u_j)(u_j-u)dx
%\end{align}
%By \eqref{3.7}, \eqref{3.11}, \eqref{n3.10} and the continuity of $M$ we have
%%\begin{align}
%M(l_p^p) \int_\Omega|\nabla u_j|^{p-2}\nabla u_j\cdot(\nabla u_j-\nabla u)dx\nonumber+ M(l_q^q)\int_\Omega a(x)|\nabla u_j|^{q-2}\nabla u_j\cdot(\nabla u_j-\nabla u)dx=o(1)
%\end{align}
%as $j\to\infty$.}
From this, we need to distinguish two situations, considering the behavior of $M$ at zero.

\begin{case}\label{step1}
{\em Let $M$ verify $M(0)=0$.}
\end{case}
\noindent 
Here, since $\ell_p\geq0$ and $\ell_q\geq0$ in \eqref{n3.10}, we split the proof in four subcases.
\begin{subcase}
{\em Let $\ell_p=0$ and $\ell_q=0$.}
\end{subcase}
\noindent
By \eqref{n3.10}, we have $\|\nabla u_j\|_p\to0$ and $\|\nabla u_j\|_{q,a}\to0$ as $j\to\infty$, implying that $u_j\to0$ in $W^{1,\mathcal H}_0(\Omega)$ thanks to \eqref{rhoh} and Proposition \ref{P2.1}. This concludes the proof in this subcase.

\begin{subcase}\label{case2}
{\em Let $\ell_p=0$ and $\ell_q>0$.}
\end{subcase}
\noindent
This situation can not occur.
Indeed, \eqref{n3.12} and $(M_2)$ yield that
\begin{equation}\label{1944}
\lim_{j\to\infty}\int_\Omega a(x)|\nabla u_j|^{q-2}\nabla u_j\cdot(\nabla u_j-\nabla u)dx=0.
\end{equation}
By \eqref{3.13}, \eqref{n3.10} and being $\ell_p=0$, we get
\begin{equation}\label{pre1944}
\lim_{j\to\infty}\int_\Omega a(x)|\nabla u|^{q-2}\nabla u\cdot(\nabla u_j-\nabla u)dx=\lim_{j\to\infty}\langle L(u),u_j-u\rangle=0.
\end{equation}
From this and \eqref{1944}, we obtain
\begin{equation}\label{1944uno}
\lim_{j\to\infty}\int_\Omega a(x)\left(|\nabla u_j|^{q-2}\nabla u_j-|\nabla u|^{q-2}\nabla u\right)\cdot(\nabla u_j-\nabla u)dx=0.
\end{equation}
Now, we recall the well known Simon inequalities, see \cite{S}, such that
\begin{equation}\label{simon}
|\xi-\eta|^\nu\leq
\begin{cases}
c\,(|\xi|^{\nu-2}\xi-|\eta|^{\nu-2}\eta)\cdot(\xi-\eta), & \mbox{if $\nu\geq2$,}\\
c\left[(|\xi|^{\nu-2}\xi-|\eta|^{\nu-2}\eta)\cdot(\xi-\eta)\right]^{\nu/2}\left(|\xi|^\nu+|\eta|^\nu\right)^{(2-\nu)/2}, & \mbox{if $1<\nu<2$,}
\end{cases}
\end{equation}
for any $\xi$, $\eta\in\mathbb R^N$, with $c$ a suitable positive constant.
Therefore, if $q\geq2$ by \eqref{simon} we have
\begin{equation}\label{1944due}
\|\nabla u_j-\nabla u\|_{q,a}^q\leq c\int_\Omega a(x)\left(|\nabla u_j|^{q-2}\nabla u_j-|\nabla u|^{q-2}\nabla u\right)\cdot(\nabla u_j-\nabla u)dx.
\end{equation}
While, if $1<q<2$ by \eqref{simon} and the H\"older inequality we have
\begin{equation}\label{1944tre}
\begin{aligned}
&\|\nabla u_j-\nabla u\|_{q,a}^q\\
&\,\,\,\,\leq c
\int_\Omega a(x)\left[\left(|\nabla u_j|^{q-2}\nabla u_j-|\nabla u|^{q-2}\nabla u\right)\cdot(\nabla u_j-\nabla u)\right]^{q/2}
\left(|\nabla u_j|^q+|\nabla u|^q\right)^{(2-q)/2}dx\\
&\,\,\,\,\leq c
\left[\int_\Omega a(x)\left(|\nabla u_j|^{q-2}\nabla u_j-|\nabla u|^{q-2}\nabla u\right)\cdot(\nabla u_j-\nabla u)dx\right]^{q/2}
\left(\|\nabla u_j\|_{q,a}^q+\|\nabla u\|_{q,a}^q\right)^{(2-q)/2}\\
&\,\,\,\,\leq \overline{c}
\left[\int_\Omega a(x)\left(|\nabla u_j|^{q-2}\nabla u_j-|\nabla u|^{q-2}\nabla u\right)\cdot(\nabla u_j-\nabla u)dx\right]^{q/2}
\end{aligned}
\end{equation}
where the last inequality follows by the boundedness of $\{u_j\}_j$ in $W^{1,\mathcal H}_0(\Omega)$, Proposition \ref{P2.1} and \eqref{rhoh}, with a suitable new positive constant $\overline{c}$. Thus, combining \eqref{1944uno}, \eqref{1944due} and \eqref{1944tre}, we obtain that $\nabla u_j\to\nabla u$ in $\left[L_a^q(\Omega)\right]^N$ as $j\to\infty$, which yields that 
\begin{equation}\label{contraddizione}
\|\nabla u\|_{q,a}=\ell_q>0
\end{equation}
by \eqref{n3.10}. By \cite[Theorem 4.9]{B}, up to a subsequence, we also obtain
\begin{equation}\label{contraddizione2}
a(x)^{1/q}|\nabla u_j(x)|\to a(x)^{1/q}|\nabla u(x)|\mbox{ a.e. in }\Omega
\end{equation}
as $j\to\infty$. 
While, \eqref{n3.10} with $\ell_p=0$ implies that $\|\nabla u_j\|_p\to0$, that is $|\nabla u_j|\to0$ in $L^p(\Omega)$ as $j\to\infty$.
Thus, going to a further subsequence, by \cite[Theorem 4.9]{B} we have $|\nabla u_j(x)|\to0$ a.e. in $\Omega$, so that also $a(x)^{1/q}|\nabla u_j(x)|\to0$ a.e. in $\Omega$ as $j\to\infty$.
From this and \eqref{contraddizione2}, we get that $a(x)^{1/q}|\nabla u(x)|=0$ a.e. in $\Omega$ which contradicts \eqref{contraddizione}.

\begin{subcase}\label{case3}
{\em Let $\ell_p>0$ and $\ell_q=0$.}
\end{subcase}
\noindent
In this subcase, \eqref{n3.12} and $(M_2)$ yield that
\begin{equation}\label{vediamo}
\lim_{j\to\infty}\int_\Omega|\nabla u_j|^{p-2}\nabla u_j\cdot(\nabla u_j-\nabla u)dx=0.
\end{equation}
By \eqref{n3.10} and Proposition \ref{P2.1}, we have $\nabla u_j\rightharpoonup\nabla u$ in $\left[L^p(\Omega)\right]^N$ as $j\to\infty$, so that
\begin{equation}\label{n3.122}
\lim_{j\to\infty}\int_\Omega |\nabla u|^{p-2}\nabla u\cdot(\nabla u_j-\nabla u)dx=0,
\end{equation}
which joint with \eqref{vediamo} gives
$$
\lim_{j\to\infty}\int_\Omega\left(|\nabla u_j|^{p-2}\nabla u_j-|\nabla u|^{p-2}\nabla u\right)\cdot(\nabla u_j-\nabla u)dx=0.
$$
From this, using \eqref{simon} and arguing as in \eqref{1944due}-\eqref{1944tre}, we obtain that $\nabla u_j\to\nabla u$ in $\left[L^p(\Omega)\right]^N$ as $j\to\infty$. Thus, by \cite[Theorem 4.9]{B}, up to a subsequence, we get 
\begin{equation}\label{vediamo3}
|\nabla u_j(x)|\to|\nabla u(x)|\mbox{ a.e. in }\Omega
\end{equation}
as $j\to\infty$. While, by \eqref{n3.10} with $\ell_q=0$ we get that $\|\nabla u_j\|_{q,a}\to0$, that is $a^{1/q}|\nabla u_j|\to0$ in $L^q(\Omega)$ as $j\to\infty$.
Thus, going to a further subsequence, by \cite[Theorem 4.9]{B} we have $a^{1/q}(x)|\nabla u_j(x)|\to0$ a.e. in $\Omega$ as $j\to\infty$, which guarantees that $|\nabla u_j(x)|\to0$ a.e. in $\Omega\setminus A$, with
$$A:=\left\{x\in\Omega:\,\,a(x)=0\right\}.
$$
Hence, by \eqref{vediamo3} we obtain that $\nabla u(x)=\overline{0}$ a.e. in $\Omega\setminus A$ so that
$$\|\nabla u_j-\nabla u\|_{q,a}^q=\int_{\Omega\setminus A}a(x)|\nabla u_j-\nabla u|^qdx=\int_{\Omega\setminus A}a(x)|\nabla u_j|^qdx
=\|\nabla u_j\|_{q,a}^q\to0
$$
as $j\to\infty$. Hence, $\nabla u_j\to\nabla u$ in $\left[L^p(\Omega)\right]^N\cap \left[L_a^q(\Omega)\right]^N$ as $j\to\infty$, thanks to \eqref{rhoh} and Proposition \ref{P2.1} we conclude that $u_j\to u$ in $W^{1,\mathcal H}_0(\Omega)$.

\begin{subcase}\label{case4}
{\em Let $\ell_p>0$ and $\ell_q>0$.}
\end{subcase}
\noindent
We can still prove \eqref{n3.12} and \eqref{n3.122}, which used in \eqref{pre1944} give    
\begin{equation}\label{n3.13}
\begin{aligned}
M&(\ell_p^p)\int_\Omega\left(|\nabla u_j|^{p-2}\nabla u_j-|\nabla u|^{p-2}\nabla u\right)\cdot(\nabla u_j-\nabla u)dx\\
&+M(\ell_q^q)\int_\Omega a(x)\left(|\nabla u_j|^{q-2}\nabla u_j-|\nabla u|^{q-2}\nabla u\right)\cdot(\nabla u_j-\nabla u)dx=o(1)
\end{aligned}
\end{equation}
as $j\to\infty$. Since by convexity we can obtain
$$
\begin{aligned}
&\left(|\nabla u_j|^{p-2}\nabla u_j-|\nabla u|^{p-2}\nabla u\right)\cdot(\nabla u_j-\nabla u)
\geq0\mbox{ a.e. in }\Omega,\\
&a(x)\left(|\nabla u_j|^{q-2}\nabla u_j-|\nabla u|^{q-2}\nabla u\right)\cdot(\nabla u_j-\nabla u)
\geq0\mbox{ a.e. in }\Omega
\end{aligned}
$$
where in the second inequality $a(x)\geq0$ a.e. in $\Omega$ by \eqref{cruciale},
then \eqref{n3.13} yields
$$\min\left\{M(\ell_p^p),M(\ell_q^q)\right\}\limsup_{j\to\infty}\langle L(u_j)-L(u),u_j-u\rangle\leq0,
$$
with both $M(\ell_p^p)>0$ and $M(\ell_q^q)>0$, thanks to $(M_2)$. Hence, by Proposition \ref{P2.4} we conclude that $u_j\to u$ in $W^{1,\mathcal H}_0(\Omega)$ as $j\to\infty$. This completes the proof of {\em Case \ref{step1}}.
\begin{case}
{\em Let $M$ verify $M(0)>0$.}
\end{case}
\noindent 
Since $M(\ell_p^p)>0$ and $M(\ell_q^q)>0$ for $\ell_p\geq0$ and $\ell_q\geq0$, thanks to also $(M_2)$, we can argue exactly as in {\em Subcase \ref{case4}}, concluding the proof of Lemma \ref{L4.3}.
\end{proof}

\begin{remark}\label{maq2}
{\em We observe that if $|A|=0$, also {\em Subcase \ref{case3}} gives a contradiction such as in {\em Subcase \ref{case2}}.
Moreover, the result in Lemma \ref{L4.3} holds assuming $(f_1')$ instead of $(f_1)$.}
\end{remark}

We are now ready to prove Theorems \ref{T3.3} and \ref{T4.4}.

\begin{proof}[Proof of Theorem \ref{T3.3}]
Since $I(0)=0$, by Lemmas \ref{L4.1}-\ref{L4.3} and the mountain pass theorem, we prove the existence of a nontrivial weak solution of \eqref{P2}.
\end{proof}

\begin{proof}[Proof of Theorem \ref{T4.4}]
Functional $I$ is even and satisfies the $(PS)$ condition thanks to $(f_3)$ and Remark \ref{maq2}, respectively.

We now prove \eqref{cond2} for $\mathcal F=I$.
For any $u\in Z_j$ with $\|u\|\geq1$, by \eqref{rhoh} and Proposition \ref{P2.1} we have
$$
\|u\|^p\leq
\begin{cases}
\|\nabla u\|_p^p+\|\nabla u\|_{q,a}^q,
&\mbox{if $\|\nabla u\|_p\geq1$ and $\|\nabla u\|_{q,a}\geq1$},\\
\|\nabla u\|_p^p+\|\nabla u\|_{q,a}^q\leq2\|\nabla u\|_p^p,
&\mbox{if $\|\nabla u\|_p\geq1$ and $\|\nabla u\|_{q,a}<1$},\\
\|\nabla u\|_p^p+\|\nabla u\|_{q,a}^q\leq2\|\nabla u\|_{q,a}^q,
&\mbox{if $\|\nabla u\|_p<1$ and $\|\nabla u\|_{q,a}\geq1$}.
\end{cases}
$$
Thus, by $(M_1)$, $(M_2)$, \eqref{3.6}, \eqref{3.1n}, \eqref{3.18}, the H\"older inequality and the definition of $\beta_j$ in \eqref{3.15}, for any $u\in Z_j$ with $\|u\|\geq 1$ we obtain
\begin{equation}\label{4.18}
\begin{aligned}
I(u)&\geq\frac{1}{p\theta}M(\|\nabla u\|_p^p)\|\nabla u\|_p^p+\frac{1}{q\theta}M(\|\nabla u\|_{q,a}^q)\|\nabla u\|_{q,a}^q
-C\|u\|_1-C\|u\|_r^r\\
&\geq\frac{\kappa}{2q\theta}\|u\|^p\!-\!\beta_j|\Omega|^{(r-1)/r}\|u\|
\!-\!\beta_j^rC\|u\|^r
\!\geq\!\left[\frac{\kappa}{2q\theta}-C\left(\beta_j|\Omega|^{(r-1)/r}\!+\!\beta_j^r\right)\|u\|^{r-p}\right]
\!\|u\|^p,
\end{aligned}
\end{equation}
with $\kappa>0$ given by $(M_2)$.
Let us choose
$$\gamma_j:=\left[\frac{\kappa}{4q\theta}\cdot
\frac{1}{C\left(\beta_j|\Omega|^{(r-1)/r}+\beta_j^r\right)}\right]^{1/(r-p)}
$$
such that $\gamma_j\to\infty$ as $j\to\infty$, since $\beta_j\to 0$ as $j\to\infty$ by \cite[Lemma 7.1]{LD} and $r>p$ by $(f'_1)$.
Then, by \eqref{4.18}, for any $u\in Z_j$ with $\|u\|=\gamma_j$ we get
$$
I(u)\geq\frac{\kappa}{4q\theta}\gamma_j^p\to\infty\mbox{ as }j\to\infty,
$$
which yields \eqref{cond2}. 

Now, let us fix $j\in\mathbb N$.
For any $u\in Y_j$ with $\|u\|\geq1$, by \eqref{3.5} and the continuity of $M$, we have
\begin{equation}\label{4.19}
\mathscr M(\|\nabla u\|_p^p)+\mathscr M(\|\nabla u\|_{q,a}^q)\leq
\begin{cases}
\mathscr M(1)\|\nabla u\|_p^{p\theta}+\mathscr M(1)\|\nabla u\|_{q,a}^{q\theta}, 
& \mbox{if $\|\nabla u\|_p\geq1$ and $\|\nabla u\|_{q,a}\geq1$,}\\
\mathscr M(1)\|\nabla u\|_p^{p\theta}+\mathcal M, 
& \mbox{if $\|\nabla u\|_p\geq1$ and $\|\nabla u\|_{q,a}<1$,}\\
\mathcal M+\mathscr M(1)\|\nabla u\|_{q,a}^{q\theta}
& \mbox{if $\|\nabla u\|_p<1$ and $\|\nabla u\|_{q,a}\geq1$,}
\end{cases}
\end{equation}
with $\displaystyle\mathcal M=\max_{t\in[0,1]}\mathscr M(t)>0$ by $(M_2)$. 
From this, by \eqref{3.4}-\eqref{3.6}, \eqref{3.17} and Proposition \ref{P2.1}, for any $u\in Y_j$ with $\|u\|\geq 1$ we have
$$
I(u)\leq\frac{\mathscr M(1)}{p}\|u\|^{q\theta}+\mathcal M-d_1c(j)\|u\|^\sigma-d_2|\Omega|,
$$
which gives \eqref{cond1} with $\rho_j>\max\{1,\gamma_j\}$ sufficiently large, since $\sigma>q\theta$ by $(f_2)$.

Thus, functional $I$ satisfies both \eqref{cond1} and \eqref{cond2}, so that \cite[Theorem 3.6]{Wil} gives the existence of an unbounded sequence of critical points of $I$ with unbounded energy, concluding the proof of Theorem \ref{T4.4}.

\end{proof}

\section*{Acknowledgments}

The authors thank Francesca Colasuonno for a nice conversation about variational double phase problems.
The authors are members of {\em Gruppo Nazionale per l'Analisi Ma\-te\-ma\-ti\-ca, la Probabilit\`a e le loro Applicazioni} (GNAMPA) 
of the {\em Istituto Nazionale di Alta Matematica} (INdAM).

A. Fiscella realized the manuscript within the auspices of the INdAM-GNAMPA project titled {\em Equazioni alle derivate parziali: problemi e modelli} (Prot\_20191219-143223-545), of the FAPESP Project titled {\em Operators with non standard growth} (2019/23917-3), of the FAPESP Thematic Project titled {\em Systems and partial differential equations} (2019/02512-5) and of the CNPq Project titled {\em Variational methods for singular fractional problems} (3787749185990982).

A. Pinamonti is partially supported by the INdAM-GNAMPA project {\em Convergenze variazionali per funzionali e operatori dipendenti da campi vettoriali}.

\end{document}